\documentclass[11pt, a4paper]{article}
\usepackage{amsmath, amsthm, amssymb, url}
\usepackage[margin=1in]{geometry}
\usepackage[english]{babel}
\usepackage{multirow}
\usepackage{authblk}
\usepackage{tikz-cd} 
\usepackage{color}
\newcommand{\End}{\mathrm{End}}
\newcommand{\Sym}{\mathrm{Sym}}

\newcommand{\id}{\mathrm{id}}

\newcommand{\Res}{\mathrm{Res}}
\newcommand{\CA}{\mathrm{CA}}

\newcommand{\GCA}{\mathrm{GCA}}

\newcommand{\Hom}{\mathrm{Hom}}
\newcommand{\Fix}{\mathrm{Fix}}

\theoremstyle{plain}

\newtheorem{corollary}{Corollary}
\newtheorem{lemma}{Lemma}
\newtheorem{proposition}{Proposition}
\newtheorem{theorem}{Theorem}

\theoremstyle{definition}
\newtheorem{definition}{Definition}

\newtheorem{example}{Example}
\newtheorem{remark}{Remark}
\newtheorem{question}{Question}

\begin{document}

\title{Further results on generalized cellular automata}
\author[1]{Alonso Castillo-Ramirez\footnote{Email: alonso.castillor@academicos.udg.mx}}
\author[1]{Luguis de los Santos Ba\~nos\footnote{Email: luguis.banos@academicos.udg.mx}}
\affil[1]{Centro Universitario de Ciencias Exactas e Ingenier\'ias, Universidad de Guadalajara, M\'exico.}

\maketitle

\begin{abstract}
Given a finite set $A$ and a group homomorphism $\phi : H \to G$, a \emph{$\phi$-cellular automaton} is a function $\mathcal{T} : A^G \to A^H$ that is continuous with respect to the prodiscrete topologies and $\phi$-equivariant in the sense that $h \cdot \mathcal{T}(x) = \mathcal{T}( \phi(h) \cdot x)$, for all $x \in A^G, h \in H$, where $\cdot$ denotes the shift actions of $G$ and $H$ on $A^G$ and $A^H$, respectively. When $G=H$ and $\phi = \id$, the definition of $\id$-cellular automata coincides with the classical definition of cellular automata. The purpose of this paper is to expand the theory of $\phi$-cellular automata by focusing on the differences and similarities with their classical counterparts. After discussing some basic results, we introduce the following definition: a $\phi$-cellular automaton $\mathcal{T} : A^G \to A^H$ has the \emph{unique homomorphism property} (UHP) if $\mathcal{T}$ is not $\psi$-equivariant for any group homomorphism $\psi : H \to G$, $\psi \neq \phi$. We show that if the difference set $\Delta(\phi, \psi)$ is infinite, then $\mathcal{T}$ is not $\psi$-equivariant; it follows that when $G$ is torsion-free abelian, every non-constant $\mathcal{T}$ has the UHP. Furthermore, inspired by the theory of classical cellular automata, we study $\phi$-cellular automata over quotient groups, as well as their restriction and induction to subgroups and supergroups, respectively.   \\

\textbf{Keywords:} cellular automata; group homomorphism; quotient group; restriction and induction.       
\end{abstract}

\section{Introduction}\label{intro}

For any group $G$ and any set $A$, a cellular automaton (CA) is a transformation of the configuration space $A^G$, consisting on all functions $x : G \to A$, that is determined by a fixed local rule that depends on a finite memory set. CA are used in discrete complex systems modeling, and make an appearance in several areas of mathematics, such as symbolic dynamics \cite{LM95}, where they are also known as \emph{sliding block codes}. Recently, the theory of CA has gained solid foundations due to its connections with group theory and topology (see the highly cited book \cite{CSC10}). 

The definition of generalized cellular automata (GCA) proposed in \cite{GCA} allows us to consider transformations between different configuration spaces $A^G$ and $A^H$, where $H$ is another arbitrary group, via a group homomorphism $\phi : H \to G$. Specifically, a \emph{$\phi$-cellular automaton} is a function $\mathcal{T} : A^G \to A^H$ such that there is a finite subset $T \subseteq G$, called a \emph{memory set} of $\mathcal{T}$, and a \emph{local function} $\mu : A^T \to A$ satisfying
\begin{equation}\label{defGCA}
\mathcal{T}(x)(h) = \mu (( \phi(h^{-1}) \cdot x) \vert_{T}),  \quad \forall x \in A^G, h \in H,
\end{equation}
where $\cdot$ is the \emph{shift action} of $G$ on $A^G$ defined by
\[ (g \cdot x)(k) := x(g^{-1}k), \quad \forall x \in A^G, g,k \in G. \]

The classical definition of CA may be recovered from the above definition by letting $G=H$ and $\phi = \id$. It was shown in \cite{GCA} that GCA share many similar properties with CA. For example, both CA and GCA are continuous functions in the \emph{prodiscrete topologies} of $A^G$ and $A^H$, which are the product topologies of the discrete topology of $A$. CA are $G$-equivariant in the sense that they commute with the shift action, while a $\phi$-cellular automaton $\mathcal{T} : A^G \to A^H$ is \emph{$\phi$-equivariant} in the following sense:
\[ h \cdot \mathcal{T}(x) = \mathcal{T}( \phi(h) \cdot x), \quad \forall x \in A^G, h \in H.  \]
This notion of $\phi$-equivariance resembles the notion of \emph{group action homomorphism} \cite[Def. 2.1]{GLM}, except that in the latter, the direction of the group homomorphism is reversed (i.e., a group action homomorphism from $A^G$ to $A^H$ is a function $f : A^G \to A^H$ and a group homomorphism $\varphi : G \to H$ such that $f(g \cdot x) = \varphi(g) \cdot f(x)$, for all $g \in G$, $x \in A^G$). A generalized version of the Curtis-Hedlund Theorem proved in \cite{GCA} characterizes $\phi$-cellular automata as the continuous $\phi$-equivariant functions from $A^G$ to $A^H$.

The goal of this paper is to expand the basic theory of generalized cellular automata by focusing on the differences and similarities with their classical counterparts. In Section 2, we discuss some basic results on the theory of GCA. In Lemma \ref{lemma:GCA-tau_phi}, we observe that for every $\phi$-cellular automaton $\mathcal{T} : A^G \to A^H$ there exists a unique $\id$-cellular automaton $\tau : A^G \to A^G$ such that 
\[ \mathcal{T} = \phi^* \circ \tau, \]
where $\phi^* : A^G \to A^H$ is the $\phi$-cellular automaton defined by $\phi^*(x) = x \circ \phi$, for all $x \in A^G$. We also show that a group $G$ is \emph{GCA-surjunctive}, in the sense that, for every finite $A$ and every $\phi \in \End(G)$, every injective $\phi$-cellular automaton $\mathcal{T} : A^G \to A^G$ is surjective, if and only if $G$ is Hopfian and surjunctive (in the classical sense for CA). This section finishes with the observation that a subgroup $K \leq G$ is fully invariant (i.e., $\phi(K) \subseteq K$, for all $\phi \in \End(G)$) if and only if the set of all $\phi$-cellular automata from $A^G$ to $A^G$ that admit a memory set contained in $K$ is closed under composition.

In Section 3, we consider the following question: if $\psi, \phi : H \to G$ are group homomorphisms and $\tau : A^G \to A^G$ is an $\id$-cellular automaton, when does $\phi^* \circ \tau = \psi^* \circ \tau$ imply $\phi = \psi$? Inspired by this, we say that a $\phi$-cellular automaton $\mathcal{T} : A^G \to A^H$ has the \emph{unique homomorphism property} (UHP) if $\mathcal{T}$ is not $\psi$-equivariant for every $\psi : H \to G$, $\psi \neq \phi$. It has been already shown in \cite{GCA} that injective GCA have the UHP. The main result of this section is the following (see Theorem \ref{th-main}). 

\begin{theorem}\label{main-intro}
If the difference set between $\phi$ and $\psi$, defined by
\[ \Delta(\phi, \psi) := \{ \psi(h)^{-1} \phi(h) : h \in H \} \subseteq G, \]
is infinite, then $\phi^* \circ \tau \neq \psi^* \circ \tau$ for every non-constant $\id$-cellular automaton $\tau : A^G \to A^G$. Hence, if $G$ is a  torsion-free abelian group, then every non-constant $\phi$-cellular automaton $\mathcal{T} : A^G \to A^H$ has the UHP. 
\end{theorem}

We finish Section 2 by proving the converse of Theorem \ref{main-intro} when $G$ is abelian or locally finite (Proposition \ref{prop-last}). 

\begin{theorem}
If $G$ is abelian or locally finite, then $\phi^* \circ \tau \neq \psi^* \circ \tau$ for every non-constant $\id$-cellular automaton $\tau : A^G \to A^G$ if and only if the set $\Delta(\phi, \psi)$ is infinite. 
\end{theorem}

In Section 4, we study $\phi$-cellular automata over quotient groups. For $\phi \in \End(G)$, we prove that if $N$ is a normal $\phi$-invariant subgroup of $G$, then there exists a unique $\hat{\phi}$-cellular automaton $\widehat{\mathcal{T} } : A^{G/N} \to A^{G/N}$, with $\hat{\phi}(gN):=\phi(g)N$, such that the following diagram commutes: 
\[ \begin{tikzcd}
A^{G/N} \arrow{r}{\widehat{\mathcal{T} }} \arrow[swap]{d}{\rho^*} & A^{G/N} \arrow{d}{\rho^*} \\%
A^G \arrow{r}{\mathcal{T}}& A^G
\end{tikzcd}\]
where $\rho : G \to G/N$ is the canonical projection and $\rho^*(x) := x \circ \rho$, for all $x \in A^G$. When $N$ is fully invariant, the map $\mathcal{T} \to \widehat{\mathcal{T}}$ is a monoid homomorphism. 

In Section 5, we study induction and restriction for $\phi$-cellular automata. For a subgroup $K \leq H$, the restriction of a $\phi$-cellular automaton $\mathcal{T}: A^G\to A^{H}$ whose memory set is contained in $\phi(K)$ is the unique $\phi \vert_K^{\phi(K)}$-cellular automaton  $\mathcal{T}_K : A^{\phi(K)} \to A^K$, where $\phi \vert_K^{\phi(K)} : K \to \phi(K)$ is the domain and codomain restriction of $\phi$, such that the following diagram commutes
\[ \begin{tikzcd}
A^{G} \arrow{r}{\mathcal{T}} \arrow[swap]{d}{\Res_{\phi(K)}} & A^{H} \arrow{d}{\Res_K} \\%
A^{\phi(K)} \arrow{r}{\mathcal{T}_K}& A^K
\end{tikzcd}\]
where $\Res_{\phi(K)}$ and $\Res_K$ are the natural restriction maps (e.g., $\Res_K(x) = x \vert_K$, for all $x \in A^H$). Induction is defined analogously for $\phi \vert_K^{\phi(K)}$-cellular automata. In both cases, when $G=H$ and $\phi=\id$, our definitions coincide with the definitions of restriction and induction for CA. We determine how restriction behaves in the composition of GCA, and prove that $\mathcal{T}$ is injective, or bijective, if and only if $\mathcal{T}_K$ is injective, or bijective, and $\phi$ is surjective, or bijective, respectively. However, it is an open question whether there is any connection between the surjectivity of $\mathcal{T}$ and $\mathcal{T}_K$.

%%%%%%%%%%%%%%%%%%%%%%%%%%%%%%%%%%%%%%%%%%%%%%%%%%%

\section{Basic results} 

We assume that the reader knows the fundamentals of group theory, topology and the theory of classical cellular automata over groups (see \cite[Ch. 1]{CSC10}). 

For the rest of the paper, let $A$ be a finite set, and let $G$ and $H$ be groups. As the case $\vert A \vert = 1$ is trivial and degenerate, we shall assume that $\vert A \vert \geq 2$ and that $\{ 0,1 \} \subseteq A$. Denote by $\Hom(H,G)$ the set of all group homomorphisms from $H$ to $G$. Define $\End(G) := \Hom(G,G)$.  

Denote by $\CA(A^G)$ and $\GCA(A^G,A^H)$ the set of all $\id$-cellular automata $\tau : A^G \to A^G$ and the set of all $\phi$-cellular automata $\mathcal{T} : A^G \to A^H$, for any $\phi \in \Hom(H,G)$, respectively. Define $\GCA(A^G) := \GCA(A^G, A^G)$. 

The following theorem summarizes the three main results obtained in \cite{GCA}.

\begin{theorem}\label{th-old}
\begin{enumerate}
\item  A function $\mathcal{T}: A^G \to A^H$ is a $\phi$-cellular automaton if and only if $\mathcal{T}$ is continuous in the prodiscrete topologies and $\phi$-equivariant.

\item Let $\mathcal{T} : A^G \to A^H$ be a $\phi$-cellular automaton with memory set $T \subseteq G$ and let $\mathcal{S} : A^H \to A^K$ be a $\psi$-cellular automaton with memory set $S \subseteq H$. Then, $\mathcal{S}\circ\mathcal{T}:A^G\to A^K$ is a $(\phi\circ\psi)$-cellular automaton with memory set $\phi(S)T := \{ \phi(s)t : s \in S, t \in T \}$.

\item A $\phi$-cellular automaton $\mathcal{T} : A^G \to A^H$ is invertible (in the sense that there exists a $\psi$-cellular automaton $\mathcal{S} : A^H \to A^G$ such that $\mathcal{T} \circ \mathcal{S}= \id_{A^H}$ and $\mathcal{S}  \circ \mathcal{T}  = \id_{A^G}$) if and only if $\mathcal{T} $ is bijective. 
\end{enumerate}
\end{theorem}

Part (2.) of Theorem \ref{th-old}, implies that the set $\GCA(A^G)$, equipped with composition, is a monoid. Part (3.) of Theorem \ref{th-old} heavily depends on the following lemma. 

\begin{lemma}\label{le-phi}
Let $\mathcal{T}  : A^G \to A^H$ be a $\phi$-equivariant function, for some $\phi \in \Hom(H,G)$.
\begin{enumerate}
 \item If $\mathcal{T} $ is surjective, then $\phi$ is injective.
 \item If $\mathcal{T} $ is injective, then $\phi$ is surjective.
 \end{enumerate}
\end{lemma}

The converse of Lemma \ref{le-phi} is clearly not true; for example, when $G=H$, there are $\id$-cellular automata that are neither surjective nor injective. As shown in \cite{GCA}, the proof of part (1.) of this lemma is straightforward, while the proof of part (2.) uses a more elaborate argument involving the periodic configurations of $A^G$ and $A^H$. 

For every $\phi \in \Hom(H,G)$, define $\phi^{*} : A^G \to A^H$ by 
\[ \phi^*(x) := x \circ \phi, \quad \forall x \in A^G. \]
Observe that for any $h \in H$,
\[ \phi^*(x)(h) = x \circ \phi(h) = (\phi(h^{-1}) \cdot x)( e ) = \id_A (( \phi(h^{-1}) \cdot x)\vert_{ \{ e \}} ) , \]
where $e$ is the identity of $G$. Hence, $\phi^*$ is a $\phi$-cellular automaton with memory set $\{ e \}$ and local function $\id_A : A^{\{ e \}} \to A$ defined by $\id_A( x) := x(e)$, for all $x \in A^{ \{ e\}}$.

The following result summarizes many of the properties satisfied by $\phi^* : A^G \to A^H$.

\begin{lemma}\label{le-star}
Let $G$ and $H$ be groups and consider $\phi \in \Hom(H,G)$. 
\begin{enumerate}
\item For any $\psi \in \Hom(H,G)$, we have
\[ \phi^{\ast}=\psi^{\ast} \quad \Leftrightarrow \quad \phi=\psi. \]

\item $\phi$ is surjective if and only if $\phi^{\ast}$ is injective. 

\item $\phi$ is injective if and only if $\phi^{\ast}$ is surjective. 

\item For any group $K$ and any $\psi \in \Hom(K,H)$, we have 
\[ (\psi \circ \phi)^* = \phi^* \circ \psi^*. \]
\end{enumerate}
\end{lemma}
\begin{proof}
\begin{enumerate}
\item If $\phi = \psi$, it is clear that $\phi^* = \psi^*$. Conversely, assume that $\phi^* = \psi^*$. Hence, for all $x \in A^G$, $h \in H$,
\[ \phi^*(x)(h) = x(\phi(h)) = x(\psi(h)) = \psi^*(x)(h). \]
For each $g \in G$, define $\chi_g \in A^G$ by $\chi_g(k) = \delta_{g,k}$, where $\delta_{g,k}$ is the Kronecker delta function. Therefore, for all $h \in H$,
\[  1 = \chi_{\phi(h)}(\phi(h)) = \chi_{\phi(h)}(\psi(h)). \]
This implies that $\phi(h)=\psi(h)$, for all $h \in H$, so $\phi=\psi$.

\item Suppose that $\phi$ is surjective. Assume that $\phi^*(x) = \phi^*(y)$ for some $x, y \in A^G$. This implies that $x(\phi(h)) = y(\phi(h))$ for all $h \in H$. As $\phi$ is surjective, this implies that $x=y$.

Conversely, suppose that $\phi$ is not surjective, so there exists $g \in G - \text{Im}(\phi)$. Define $x, y \in A^G$ by $x(a) = y(a)$ for all $a \in G - \{g\}$, and $x(g)=1$ and $y(g)=0$. Clearly, $x \neq y$, but for all $h$,
\[ \phi^*(x)(h) = x(\phi(h)) = y(\phi(h)) = \phi^*(y)(h). \] 
Hence, $\phi^*(x) = \phi^*(y)$, which shows that $\phi^*$ is not injective. 

\item Suppose that $\phi$ is injective. We shall show that any $y\in A^{H}$ has a preimage under $\phi^*$. We define $x \in A^G$ as follows 
\[ x (g) := \begin{cases}
y(h) & \text{if } \exists h \in H \text{ such that } g = \phi(h) \\
0    & \text{ otherwise. }
\end{cases} \] 
The function $x : G \to A$ is well-defined as $\phi$ is injective, so every $g \in G$ has at most one preimage in $H$ under $\phi$. Observe that, for any $h \in H$, $\phi^{\ast}(x)(h) = x(\phi(h)) = y(h)$. Therefore, $\phi^*$ is surjective.

Conversely, suppose that $\phi$ is not injective. Then, there exist $h_1, h_2 \in H$, $h_1 \neq h_2$, such that $\phi(h_1) = \phi(h_2)$. This implies that for every $x \in A^G$, 
\[  \phi^*(x)(h_1) = x(\phi(h_1)) = x(\phi(h_2)) = \phi^*(x)(h_2). \]
Therefore, any configuration $y \in A^H$ with $y(h_1) \neq y(h_2)$ is outside the image of $\phi^*$, so the latter is not surjective. 

\item For any $x \in A^G$ and $h \in H$ we have
\[ (\psi \circ \phi)^*(x) = x \circ \psi \circ \phi = (\psi^*(x)) \circ \phi = (\phi^* \circ \psi^* )(x). \] 

\end{enumerate}
\end{proof}

As in the case of classical CA, $\phi$-cellular automata do not have a unique memory set in general. If a $\phi$-cellular automaton $\mathcal{T} : A^G \to A^H$ has memory set $T \subseteq G$ and local function $\mu : A^T \to A$, then, for any finite subset $T^\prime$ of $G$ such that $T \subseteq T^\prime$, we may define $\mu^\prime : A^{T^\prime} \to A$ by $\mu^\prime(x) := \mu(x \vert_T)$, for all $x \in A^{T^\prime}$. The memory set $T^\prime$ and local function $\mu^\prime$ also define $\mathcal{T}$ since
\[  \mu^{\prime}( ( \phi(h^{-1}) \cdot x)\vert_{T^\prime}) =  \mu( ( \phi(h^{-1}) \cdot x)\vert_{T}) =  \mathcal{T}(x)(h), \quad \forall x \in A^G, h \in H. \] 

The following result is analogous to \cite[Lemma 1.5.1]{CSC10}, but we include it here for completeness. 

\begin{lemma}
Let $\mathcal{T} : A^G \to A^H$ be a $\phi$-cellular automaton. If $S_1$ and $S_2$ are memory sets for $\mathcal{T}$, then $S_1 \cap S_2$ is also a memory set for $\mathcal{T}$. Hence, $\mathcal{T}$ admits a unique memory set of minimal cardinality, which is the intersection of all memory sets admitted by $\mathcal{T}$. 
\end{lemma}
\begin{proof}
Let $\mu_1 : A^{S_1} \to A$ and $\mu_2 : A^{S_2} \to A$ be the local functions defining $\mathcal{T}$. This implies that
\begin{equation} \label{mu}
 \mu_1(x \vert_{S_1}) = \mathcal{T}(x)(e) = \mu_2(x \vert_{S_2}), \quad \forall x \in A^G. 
 \end{equation}
Define $\mu : A^{S_1 \cap S_2} \to A$ by 
\[ \mu(x) := \mu_1(\hat{x} \vert_{S_1}), \quad \forall x \in A^{S_1 \cap S_2}, \]
where $\hat{x} \in A^{G}$ is any extension of $x \in A^{S_1 \cap S_2}$. To see that $\mu$ this is well-defined, let $\hat{x}_1, \hat{x}_2 \in A^{G}$ be such that $\hat{x}_1 \vert_{S_1 \cap S_2} = \hat{x}_2 \vert_{S_1 \cap S_2}$. Define $z \in A^{G}$ as follows
\[  z(g) := \begin{cases}
\hat{x}_1(g) & \text{ if } g \in S_1\\
\hat{x}_2(g) & \text{ if } g \in G \setminus S_1
\end{cases}\] 
Note that $z \vert_{S_1} = \hat{x}_1 \vert_{S_1}$ and $z \vert_{S_2} = \hat{x}_2 \vert_{S_2}$. By (\ref{mu}), we have
\[ \mu_1(\hat{x}_1 \vert_{S_1}) =\mu_1(z \vert_{S_1}) = \mu_2(z \vert_{S_2}) = \mu_2(\hat{x}_2 \vert_{S_2}) = \mu_1(\hat{x}_2 \vert_{S_1}) .  \]
Clearly, $\mu : A^{S_1 \cap S_2} \to A$ is a local function defining $\mathcal{T}$, so $S_1 \cap S_2$ is a memory set for $\mathcal{T}$
 \end{proof}
 
 The unique memory set of minimal cardinality of $\mathcal{T}$ is called the \emph{minimal memory set} of $\mathcal{T}$. 

\begin{lemma}\label{lemma:GCA-tau_phi}
For any $\phi$-cellular automaton $\mathcal{T} : A^G \to A^H$ there exists a unique $\tau \in \CA(A^G)$ such that 
\[ \mathcal{T}= \phi^* \circ \tau, \]
i.e., the following diagram commutes
\[
\begin{tikzcd}
 & A^G \arrow[dr,"\phi^*"] \\
A^G \arrow[ur,"\tau"] \arrow[rr,"\mathcal{T}"] && A^H
\end{tikzcd}
\]
Furthermore, the minimal memory set of $\mathcal{T}$ is equal to the minimal memory set of $\tau$. 
\end{lemma}
\begin{proof}
Let $T$ be a memory set of $\mathcal{T}$ and let $\mu:A^T\to A$ be the local function of $\mathcal{T}$. Consider the $\id$-cellular automaton $\tau: A^{G}\to A^{G}$ defined by $\mu : A^T \to A$ as local function: 
\[ \tau (x)(g) := \mu((g^{-1} \cdot x)|_{T}), \]
for all $x\in A^{G}$ and $g\in G$. Observe that 
\[ \phi^\ast \circ \tau(x)(h) = \tau(x)(\phi(h)) =  \mu((\phi(h)^{-1} \cdot x)|_{T}) = \mathcal{T}(x)(h), \]
for all $x \in A^G$, $h \in H$. Therefore, $\mathcal{T}= \phi^* \circ \tau$. The above equality implies that $T$ is a memory set of $\mathcal{T}$ if and only if $T$ is a memory set of $\tau$, so they must have the same minimal memory set. 

Finally, in order to show the uniqueness of $\tau$, let $\sigma \in \CA(A^G)$ be such that $\phi^* \circ \tau = \phi^{*} \circ \sigma$. Let $S$ be a memory set of $\sigma$ and let $\mu^\prime : A^S \to A$ be the local function of $\sigma$. Then, for all $x \in A^G$,
\[ \mu(x \vert_T) = \phi^* \circ \tau(x)(e) = \phi^* \circ \sigma(x)(e)= \mu^\prime(x \vert_{S}). \]
This implies that for all $x \in A^G$, $g \in G$, 
\[ \tau(x)(g) = \mu((g^{-1} \cdot x)\vert_T) = \mu^\prime((g^{-1} \cdot x)\vert_S) = \sigma(x)(g). \]
Therefore, $\tau = \sigma$. 
\end{proof}

A group $G$ is \emph{surjunctive} if, for every finite set $A$, every injective cellular automaton $\tau : A^G \to A^G$ is surjective (\cite[Sec. 3.1]{CSC10}). The famous Gottschalk conjecture states that all groups are surjunctive; one of the most important steps towards proving it is the Gromov-Weiss Theorem \cite[Theorem 7.8.1.]{CSC10}, which establishes that every sofic group is surjunctive. The following is a natural analog of surjunctivity for GCA. 

\begin{definition}
A group $G$ is \emph{GCA-surjunctive} if, for every finite set $A$, every injective $\mathcal{T} \in \GCA(A^G)$ is surjective. 
\end{definition}

Recall that a group $G$ is \emph{Hopfian} is every surjective $\phi \in \End(G)$ is injective \cite[Sec. 2.4]{CSC10}. Examples of Hopfian groups are finitely generated residually finite groups (e.g. $G$ is $\mathbb{Z}^d$, with $d \geq 1$, or $G$ is a free group $F_n$, with $n \geq 1$) and the additive group $\mathbb{Q}$ \cite[Example 2.4.2.(b)]{CSC10}. Groups that are known to be non-Hopfian are the additive group $\mathbb{R}$ \cite[Exercise 2.23]{ECAG} and the Baumslag-Solitar group $B(2,3)$ (see \cite{BS62}, or \cite[Exercise 2.13]{ECAG} for a self-contained presentation).

\begin{proposition}
A group $G$ is GCA-surjunctive if and only if $G$ is Hopfian and surjunctive. 
\end{proposition}
\begin{proof}
Suppose first that $G$ is GCA-surjunctive. It is clear that $G$ is surjunctive because every classical cellular automaton is an $\id$-cellular automaton. Let $\phi \in \End(G)$ be surjective. By Lemma \ref{le-star} (2.), $\phi^* \in \GCA(A^G)$ is injective, so it must be also surjective as $G$ is GCA-surjunctive. Now by Lemma \ref{le-star} (3.), $\phi$ is injective, which shows that $G$ is Hopfian. 

Suppose now that $G$ is Hopfian and surjunctive. For any finite set $A$, let $\mathcal{T}  : A^G \to A^G$ be an injective $\phi$-cellular automaton, with $\phi \in \End(G)$. If $\vert A \vert \leq 1$, $\mathcal{T}$ is trivially surjective, so assume that $\vert A \vert \geq 2$. By Lemma \ref{lemma:GCA-tau_phi}, there exists $\tau \in \CA(A^G)$ such that $\mathcal{T} = \phi^* \circ \tau$. Hence, the injectivity of $\mathcal{T} $ implies the injectivity of $\tau$. Now, $\tau$ is surjective by the surjunctivity of $G$. Lemma \ref{le-phi} (2.) implies that $\phi$ is surjective, so $\phi$ must be injective because $G$ is Hopfian. Therefore, $\phi^*$ and $\tau$ are both bijective, so $\mathcal{T} $ is bijective (and in particular, surjective). This shows that $G$ is GCA-surjunctive.
\end{proof}

Recall that a subgroup $K$ of $G$ is \emph{fully invariant} if $K$ is $\phi$-invariant (i.e., $\phi(K) \subseteq K$) for all $\phi \in \End(G)$. For example, it is well-known that every subgroup of a cyclic group is fully invariant, and that the commutator subgroup of any group is fully invariant. We finish this section with a characterization of fully invariant subgroups of $G$ in terms of GCA. 

\begin{proposition}
Let $K$ be a subgroup of $G$. The set
\[ \GCA_K (A^G) := \{ \mathcal{T}  \in \GCA(A^G) : \text{the minimal memory set of }\mathcal{T} \text{ is contained in }   K \} \]  
is a submonoid of $\GCA(A^G)$ if and only if $K$ is fully invariant.
\end{proposition}
\begin{proof}
Suppose that $\GCA_K(A^G)$ is a submonoid. Let $\phi \in \End(G)$ and let $\tau \in \CA(A^G)$ be with minimal memory set $\{ k \}$, for $k \in K$ (e.g., take $\tau(x)(g):=x(gk)$, for all $x \in A^G$, $g \in G$). By Theorem \ref{th-old} (2.),  $\tau \circ \phi^*$ is a $\phi$-cellular automaton with memory set $\phi(\{ k \}) = \{ \phi(k) \}$. By hypothesis, we have $\tau \circ \phi^* \in \GCA_K(A^G)$, so $\phi(k) \in K$. As $\phi$ and $k$ were arbitrary, it follows that $K$ is fully invariant.

Conversely, assume that $K$ is fully invariant. Let $\mathcal{T}$ and $\mathcal{S}$ be a $\phi$- and $\psi$-cellular automaton with memory sets $T \subseteq K$ and $S \subseteq K$, respectively. By Theorem \ref{th-old} (2.), $\mathcal{T} \circ \mathcal{S}$ has memory set $\phi(T)S$. By $\phi$-invariance, $\phi(T) \subseteq K$, so $\phi(T)S \subseteq K$. This shows that $\mathcal{T} \circ \mathcal{S} \in \GCA_K(A^G)$, and $\GCA_K(A^G)$ is a submonoid. 
\end{proof}

%%%%%%%%%%%%%%%%%%%%%%%%%%%%%%%%%%%%%%%%%%%

%%%%%%%%%%%%%%%%%%%%%%%%%%%%%%%%%%%%%%%%%%%%%%

\section{Homomorphisms of GCA}\label{Section:Homomorphisms}

A GCA $\mathcal{T} : A^G \to A^H$ may be a $\phi$-cellular automaton and a $\psi$-cellular automaton for different $\phi, \psi \in \Hom(H,G)$. Constant GCA are an extreme example of this situation. 

\begin{lemma}\label{le-constant}
Let $\mathcal{T} : A^G \to A^H$ be a constant $\phi$-cellular automaton. Then $\mathcal{T}$ is a $\psi$-cellular automaton for every $\psi \in \Hom(H,G)$. 
\end{lemma}
\begin{proof}
Let $T$ be a memory set and let $\mu : A^T \to A$ be the local function of $\mathcal{T}$. As $\mathcal{T}$ is constant, then $\mu$ must be constant, so for any $\psi \in \Hom(H,G)$, $h \in H$, $x \in A^G$,  
\[ \mathcal{T}(x)(h) = \mu( (\phi(h) \cdot x) \vert_T)  = \mu( (\psi(h) \cdot x) \vert_T).  \]
The result follows. 
\end{proof}

By Lemma \ref{lemma:GCA-tau_phi}, if $\mathcal{T}$ is a $\phi$- and $\psi$-cellular automaton, there exist $\tau, \sigma \in \CA(A^G)$ such that 
\[ \mathcal{T} = \phi^* \circ \tau = \psi^* \circ \sigma.  \]
As $\tau$ and $\sigma$ are defined by using the local function of $\mathcal{T}$, we must have $\tau = \sigma$. Therefore, in this section we are interested in the following question: for $\tau \in \CA(G;A)$ and $\phi, \psi \in \Hom(H,G)$, when does $\phi^* \circ \tau = \psi^* \circ \tau$ imply $\phi = \psi$?

\begin{definition}
We say that $\mathcal{T} \in \GCA(A^G,A^H)$ has the \emph{unique homomorphism property} (UHP) if $\mathcal{T} =\phi^* \circ \tau = \psi^* \circ \tau$, for $\phi, \psi \in \Hom(H,G)$, $\tau \in \CA(A^G)$, implies $\phi = \psi$. 
 \end{definition}

The following result is \cite[Lemma 9]{GCA}, but we shall include it here for completeness.

\begin{lemma}\label{le-q2-1}
Every injective $\mathcal{T} \in \GCA(A^G,A^H)$ has the UHP.
\end{lemma}
\begin{proof}
Suppose that $\mathcal{T} = \phi^* \circ \tau = \psi^* \circ \tau$ for some $\phi, \psi \in \Hom(H,G)$, and $\tau \in \CA(G;A)$. Then $\mathcal{T}$ is both $\phi$- and $\psi$-equivariant, so for all $h \in H$, $x \in A^G$,
\[ \tau( \phi(h) \cdot x) = h \cdot \tau(x) = \tau( \psi(h) \cdot x).  \]
As $\tau$ is injective, then $\phi(h) \cdot x = \psi(h) \cdot x$ for all $h \in H$, $x \in A^G$. As $\vert A \vert \geq 2$, the shift action is faithful (see \cite[2.7.2]{CSC10}), so we have $\phi(h) = \psi(h)$ for all $h \in H$. Then $\phi = \psi$ and the result follows. 
\end{proof}

It is obvious that if $\tau \in \CA(A^G)$ is surjective, then, for every $\phi \in \Hom(H,G)$, $\mathcal{T} = \phi^* \circ \tau$ has the UHP because $\tau$ is right-cancellative. In order to generalize this criterion for UHP, we shall introduce some notation. 

\begin{definition}
We say that $\chi_{g}^a \in A^G$ is an \emph{$a$-characteristic function} on $g \in G$ and $a \in A$ if 
\[ \chi_{g}^a(h)=a  \ \Leftrightarrow \ h=g. \]  
\end{definition}

\begin{example}
Let $G= \mathbb{Z}$. We represent a configuration $x \in A^\mathbb{Z}$ as the bi-infinite sequence $\dots x(-2) x(-1) . x(0) x(1) x(2) \dots$. Observe that if $A=\{0,1\}$, then $A^\mathbb{Z}$ has a unique $1$-characteristic function on $0 \in \mathbb{Z}$:
\[ \chi_0^1 = \dots 0 0 . 1 0 0 \dots \]
However, if $A=\{ 0,1,2 \}$, then $A^\mathbb{Z}$ has infinitely many $1$-characteristic functions on $0 \in \mathbb{Z}$:
\[ \chi_0^1 = \dots a_{-2} a_{-1} . 1 a_{1} a_{2} \dots  \]
where $a_k \in \{0,2\}$, $k \in \mathbb{Z}$. 
\end{example}

\begin{lemma}\label{le-specialconfig}
Consider $\tau\in \CA(A^G)$ and suppose that $\chi_g^a \in \tau(A^{G})$, where $\chi_g^a$ is an $a$-characteristic function on $g \in G$. Then, for all $k \in G$, $\tau(A^{G})$ contains an $a$-characteristic function $\chi_{k}^a$ on $k$.
\end{lemma}
\begin{proof} 
For any $k \in G$, we claim that $\chi_k^a := kg^{-1} \cdot \chi_g^a$ is an $a$-characteristic function on $k$. Observe that for any $h \in G$,
\[ \left( kg^{-1} \cdot \chi_g^a \right)(h) = a \quad \Leftrightarrow \quad  \chi_g^a (gk^{-1} h) = a \quad \Leftrightarrow \quad gk^{-1} h = g \quad \Leftrightarrow \quad h=k. \]
This shows that $\chi_k^a(h)=a$ if and only if $h=k$, so $\chi_k^a$ is an $a$-characteristic function on $k$. Finally, it follows by the $G$-equivariance of $\tau$ that if $\chi_g^a \in \tau(A^{G})$ then $\chi_k^a = kg^{-1} \cdot \chi_g^a \in \tau(A^G)$.
\end{proof}
 
\begin{proposition}
Let $\mathcal{T} := \phi^* \circ \tau$ be with $\tau \in \CA(A^G)$, $\phi \in \Hom(H,G)$. If $\tau(A^{G})$ has an $a$-characteristic function on $g$, for some $g \in G$ and $a \in A$, then $\mathcal{T}$ has the UHP.
\end{proposition}

\begin{proof}
Suppose that $\phi^* \circ \tau =\psi^* \circ \tau$ for some $\psi \in \Hom(H,G)$. Let $\chi_{g}^a \in \tau(A^{G})$ be an $a$-characteristic function on $g \in G$. By Lemma \ref{le-specialconfig}, there exists an $a$-characteristic function $\chi_{\phi(h)}^a \in \tau(A^G)$ for any $h \in H$. Let $y_h \in A^{G}$ be such that $\tau (y_h)= \chi_{\phi(h)}^a$. Since $\phi^* \circ \tau =\psi^* \circ \tau$, we have
\[   \phi^*(\chi_{\phi(h)}^a) = \mathcal{T}(y_h) = \psi^*(\chi_{\phi(h)}^a). \]
Evaluating at $h \in H$, we obtain
\[ a = \chi_{\phi(h)}^a (\phi(h))  = \phi^*(\chi_{\phi(h)}^a)(h) = \psi^*(\chi_{\phi(h)}^a)(h) = \chi_{\phi(h)}^a (\psi(h)).   \]
This implies that $\phi(h) = \psi(h)$. As $h \in H$ was arbitrary, we conclude that $\phi=\psi$.
\end{proof}

%%%%%%%%%%%%%%%%%%%%%%%%%%%%%%%%%%%%%%%%%%%%

%\begin{lemma}\label{le-finite}
%Let $G$ be a group and $S$ a finite subset of $G$. Suppose that $gS = S$ for some $g \in G$. Then $g$ has finite order.
%\end{lemma}
%\begin{proof}
%By hypothesis, the action of $\langle g \rangle$ on $S$ by left multiplication is well-defined. As this action is faithful, $\langle g \rangle$ is isomorphic to a subgroup of $\Sym(S)$, which is finite. The result follows. 
%\end{proof}

The main results of this section use the following elementary observation for infinite groups. 

\begin{lemma}\label{le-power}
Let $G$ be an infinite group and let $T$ be a finite subset of $G$. For any infinite subset $R \subseteq G$, there exists $r \in R$ such that $rT \cap T= \emptyset$.  
\end{lemma}
\begin{proof}
Assume that $r T\cap T\neq \emptyset$ for all $r \in R$. Hence, for all $r \in R$ there exist $t_1, t_2 \in T$ such that $r t_1 = t_2$. This implies that $r = t_2 (t_1)^{-1} \in T T^{-1}$, so $R \subseteq T T^{-1}$, where $T^{-1} := \{ t^{-1} : t \in T\}$. As $T$ is finite, then $T T^{-1}$ is finite, so $R \subseteq T T^{-1}$ contradicts the hypothesis that $R$ is infinite. 
\end{proof}

As a first application of the previous result, we have a characterization of constant GCA. 

\begin{lemma}
Suppose that $\Hom(H,G)$ contains a homomorphism with infinite image. Then, $\mathcal{T} \in \GCA(A^G,A^H)$ is constant if and only if $\mathcal{T} =\psi^* \circ \tau$, for all $\psi \in \Hom(H,G)$.
\end{lemma}
\begin{proof}
The direct implication follows from Lemma \ref{le-constant}, so assume that $\mathcal{T} =\psi^* \circ \tau$, for all $\psi \in \Hom(H,G)$. Let $T$ be a memory set and let $\mu : A^T \to A$ be the local function of $\mathcal{T}$. Let $\psi_0 : H \to G$ be the trivial homomorphism (i.e., $\psi_0(h)=e$, for all $h \in H$), and let $\psi_1 : H \to G$ be a homomorphism such that $\psi_1(H)$ is infinite. Then, for all $x \in A^G$, $h \in H$, we have
\begin{equation}\label{property}
 \mu( (\psi_1(h^{-1}) \cdot x)\vert_T ) = \psi_1^* \circ \tau(x)(h) = \psi_0^* \circ \tau(x)(h) = \mu( x \vert_T). 
\end{equation}
Suppose that $\mu$ is not constant, so there exist $z_1, z_2 \in A^T$ such that $\mu(z_1) \neq \mu(z_2)$. By Lemma \ref{le-power}, there exists $k \in \psi_1(H)$ such that $k T \cap T = \emptyset$. Define $x \in A^G$ as follows 
\[  x(g) := \begin{cases} 
z_1(g) & \text{ if } g \in T, \\
z_2(k^{-1}g) & \text{ if } g \in k T, \\ 
0 & \text{ otherwise.} 
\end{cases} \]
Note that $z_1 = x \vert_T$ and $z_2 = ( k^{-1} \cdot x) \vert_T$. By (\ref{property}), we obtain that $\mu(z_1) = \mu(z_2)$, which is a contradiction. Therefore, $\mu$ is constant, so $\mathcal{T}$ is constant.  
\end{proof}

For $\phi, \psi \in \Hom(H,G)$, define the \emph{difference set between $\phi$ and $\psi$} by
\[ \Delta(\phi, \psi) := \{ \psi(h)^{-1} \phi(h) : h \in H \} \subseteq G. \]
Note that $\Delta(\phi,\psi) = \{ e \}$ if and only if $\phi = \psi$. If $G$ is abelian, then $\Delta(\phi, \psi)$ is a subgroup of $G$. 

\begin{theorem}\label{th-main}
Let $\mathcal{T} := \phi^* \circ \tau$ be non-constant, with $\tau \in \CA(A^G)$, $\phi \in \Hom(H,G)$. If $\Delta(\phi, \psi)$ is infinite for $\psi \in \Hom(H,G)$, then $\mathcal{T} \neq \psi^* \circ \tau$. 
\end{theorem}
\begin{proof} 
Let $T \subseteq G$ be a memory set and $\mu: A^{T}\to A$ be a local function of $\mathcal{T}$. Since $\mathcal{T}$ is not constant, then $\mu$ is not constant, so there exist $z_{1},z_{2}\in A^{T}$ such that $\mu(z_{1})\neq \mu(z_{2})$. 

By Lemma \ref{le-power} applied to $\Delta(\phi, \psi)$, there exists $h \in H$ such that
\[  \psi(h)^{-1} \phi(h) T \cap T = \emptyset \quad \Rightarrow \quad   \phi(h) T \cap  \psi(h) T  = \emptyset.  \] 
Take $x \in A^G$ such that $x \vert_{\phi(h)T}=z_{1}$ and $x \vert_{\psi(h)T}=z_{2}$. This implies that $(\phi(h^{-1}) \cdot x)|_{T} = z_1$ and $(\psi(h^{-1}) \cdot x)|_{T} = z_2$. Therefore, 
\[ \phi^* \circ \tau (x)(h)=\mu ((\phi(h^{-1}) \cdot x)|_{T}) = \mu (z_1) \neq \mu(z_2) = \mu ((\psi(h^{-1}) \cdot x)|_{T})=\psi^* \circ \tau (x)(h). \]
Therefore, $\mathcal{T} =  \phi^* \circ \tau  \neq \psi^* \circ \tau$. 
\end{proof}

\begin{corollary}\label{cor-UHP}
Let $G$ be a torsion-free abelian group (for example, $G \cong \mathbb{Z}^d$, with $d \in \mathbb{N}$ or $G \cong \mathbb{Q}$). Then, every non-constant $\mathcal{T} \in \GCA(A^G,A^H)$ has the UHP. 
\end{corollary}
\begin{proof}
Let $\phi, \psi \in \Hom(H,G)$ be such that $\phi \neq \psi$. As $G$ is abelian, then $\Delta(\phi, \psi)$ is a subgroup of $G$. Since $G$ is torsion-free and $\Delta(\phi, \psi)$ is nontrivial, then $\Delta(\phi, \psi)$ must be infinite. The result follows by Theorem \ref{th-main}. 
\end{proof}

%\begin{corollary}
%Let $\phi, \psi \in \Hom(H,G)$ be such that $\phi \neq \tau$. If the set $\{ \psi(h)^{-1} \phi(h) : h \in H \}$ is torsion-free and $[\psi(h), \phi(h)] = e$, for all $h \in H$, then $\tau_\phi \neq \tau_\psi$, for all non-constant $\tau \in \CA(G;A)$. 
%\end{corollary}

%%%%%%%%%%%%%%%%%%%%%%%%%%%%%%%%

We are interested now in the converse of Theorem \ref{th-main}. Recall that a function $\mu : A^T \to A$ is \emph{symmetric} if for any $f \in \Sym(T)$, we have $\mu(z) = \mu (z \circ f)$ for all $z \in A^T$.

\begin{proposition}\label{prop}
Let $\mathcal{T} : A^G \to A^H$ be a $\phi$-cellular automaton with memory set $T \subseteq G$ and local function $\mu : A^T \to A$. Suppose that there exists $\psi \in \Hom(H,G)$ such that 
\[\psi(h)^{-1}\phi(h)T = T, \quad \forall h \in H. \]
If $\mu$ is symmetric, then $\mathcal{T}$ is also a $\psi$-cellular automaton.  
\end{proposition}
\begin{proof}
Fix $h \in H$. Define $f_h \in \Sym(T)$ by $f_h(t) := \psi(h)^{-1}\phi(h)t$, for all $t \in T$. Observe that 
 \[ ( \psi(h^{-1}) \cdot x) \vert_T  \circ f_h(t) = x( \psi(h) f_h(t)) = x( \phi(h)t) = ( \phi(h^{-1}) \cdot x) \vert_T (t) . \]
As $\mu$ is symmetric,
\[ \mu( \psi(h^{-1}) \cdot x) \vert_T ) = \mu( ( \psi(h^{-1}) \cdot x) \vert_T  \circ f_h) = \mu( \phi(h^{-1}) \cdot x) \vert_T) = \mathcal{T}(x)(h) . \]
The result follows.
\end{proof}

\begin{remark}
In Proposition \ref{prop}, the hypothesis $\psi(h)^{-1}\phi(h)T = T$ for all $h \in H$ implies that $\Delta(\phi, \psi)$ is finite (c.f. Lemma \ref{le-power}). 
\end{remark}

\begin{remark}
If $T$ is a finite subset of $G$ that is not a subgroup, we cannot conclude that $aT = T$ for all $a \in T$. Therefore, the hypothesis that $\psi(h)^{-1}\phi(h)T = T$, for all $h \in H$, in Proposition \ref{prop} may not be replaced by the assumption that $\Delta(\phi, \psi) \subseteq T$. 
\end{remark}

The following is a converse of Theorem \ref{th-main} when $G$ is abelian or locally finite (i.e., every finitely generated subgroup of $G$ is finite).

\begin{proposition}\label{prop-last}
Let $G$ be an abelian or locally finite group, and let $\phi, \psi \in \Hom(H,G)$. Then, $\phi^* \circ \tau \neq \psi^* \circ \tau$ for all non-constant $\tau \in \CA(A^G)$ if and only if $\Delta(\phi, \psi)$ is infinite. 
\end{proposition}
\begin{proof}
The converse is given by Theorem \ref{th-main}. Assume that $\Delta(\phi, \psi)$ is finite and let $T := \langle \Delta(\phi, \psi) \rangle$ (i.e., $T$ is the subgroup of $G$ generated by $\Delta(\phi, \psi)$). If $G$ is locally finite, then $T$ is a finite subgroup of $G$ by definition. If $G$ is abelian, then $\Delta(\phi, \psi)$ is itself a finite subgroup of $G$, and $T = \Delta(\phi, \psi)$.  

If $\vert A \vert =q$, equip the set $A$ with addition modulo $q$. Consider the cellular automaton $\tau : A^G \to A^G$ with memory set $T$ and local function $\mu : A^T \to A$ defined by
\[ \mu(y) := \sum_{t \in T} y(t), \quad \forall y \in A^T. \]
As $T$ is a subgroup of $G$ that contains $\Delta(\phi, \psi)$, then $\psi(h)^{-1} \phi(h)T = T$ for all $h \in H$, by closure. Furthermore, the local function $\mu$ is symmetric, so we conclude that $\phi^* \circ \tau = \psi^* \circ \tau$ by Proposition \ref{prop}. 
\end{proof}

%%%%%%%%%%%%%%%%%%%%%%%%%%%%%%%%%%%%%%%%%%%%

%%%%%%%%%%%%%%%%%%%%%%%%%%%%%%%%%%%%%%%%%%%%%%%%%%

\section{GCA over quotient groups}

In this section we generalize several of the results included in \cite[Ch. 1]{CSC10}. For a subgroup $K \leq G$, define the set of $K$-periodic configurations in $A^G$ by
\[ \Fix(K) := \{ x \in A^G : k \cdot x = x, \ \forall k \in K \}. \]

 \begin{lemma}\label{le-fix}
 Let $\phi \in \End(G)$ and let $K \leq G$ be a $\phi$-invariant subgroup. Then, $\Fix(K)$ is $\mathcal{T}$-invariant, for every $\mathcal{T} \in \GCA(A^G)$. 
 \end{lemma}
\begin{proof}
Let $x \in \Fix(K)$ and let $\mathcal{T} : A^G \to A^G$ be a $\phi$-cellular automaton. It follows by $\phi$-equivariance and the $\phi$-invariance of $K$ that 
\[ k \cdot \mathcal{T}(x) = \mathcal{T}(\phi(k) \cdot x) = \mathcal{T}(x),  \]
for all $k \in K$. This shows that $\mathcal{T}(x) \in \Fix(K)$. 
\end{proof}

For a normal subgroup $N\trianglelefteq G$, let $\rho :  G \to G/N$ be the canonical projection $\rho(g):=gN$, for all $g \in G$. The following result is well-known. 

\begin{lemma}\label{le-endGN}
Let $\phi \in \End(G)$ and let $N \trianglelefteq G$ be a normal $\phi$-invariant subgroup. Then, there exists a unique endomorphism $\hat{\phi} : G/N \to G/N$ such that the following diagram commutes
\[ \begin{tikzcd}
G \arrow{r}{\phi} \arrow[swap]{d}{\rho} & G \arrow{d}{\rho} \\%
G/N \arrow{r}{\hat{\phi}}& G/N
\end{tikzcd}\]
\end{lemma}

\begin{lemma}
Let $\phi \in \End(G)$ and let $N \trianglelefteq G$ be a normal $\phi$-invariant subgroup. For any $\phi$-cellular automaton $\mathcal{T} : A^G \to A^G$ there exists a unique $\hat{\phi}$-cellular automaton $\widehat{\mathcal{T} } : A^{G/N} \to A^{G/N}$ such that the following diagram commutes: 
\[ \begin{tikzcd}
A^{G/N} \arrow{r}{\widehat{\mathcal{T} }} \arrow[swap]{d}{\rho^*} & A^{G/N} \arrow{d}{\rho^*} \\%
A^G \arrow{r}{\mathcal{T}}& A^G
\end{tikzcd}\]
\end{lemma}
\begin{proof}
Since the canonical projection $\rho : G \to G/N$ is surjective, it follows by Lemma \ref{le-star} that $\rho^* : A^{G/N} \to A^G$ is an injective $\rho$-cellular automaton. Moreover, as $\ker(\rho)=N$, it follows by \cite[Prop. 1]{GCA} that $\rho^{*}(A^{G/N}) = \Fix(N)$. Let $\rho^*_{\downarrow} : A^{G/N} \to \Fix(N)$ be the codomain restriction of $\rho^*$. 

For any $\phi$-cellular automaton $\mathcal{T} : A^G \to A^G$, Lemma \ref{le-fix} shows that $\mathcal{T}(\Fix(N)) \subseteq \Fix(N)$, so we may define $\widehat{\mathcal{T} } : A^{G/N} \to A^{G/N}$ by 
\[ \widehat{\mathcal{T} }  := (\rho^*_{\downarrow})^{-1} \circ (\mathcal{T})\vert_{\Fix(N)} \circ \rho^*_{\downarrow} . \] 
In order to show that $\widehat{\mathcal{T} }$ is a $\hat{\phi}$-cellular automaton, let $\tau \in \CA(A^G)$ be such that $\mathcal{T} = \phi^* \circ \tau$. By Lemmas \ref{le-star} and \ref{le-endGN}, 
\[ \phi^* \circ \rho^* = \rho^* \circ (\hat{\phi})^*, \]
and restricting to $\Fix(N)$ we see that
\[ (\rho^*_{\downarrow})^{-1} \circ \phi^* \vert_{\Fix(N)} = (\hat{\phi})^* \circ (\rho^*_{\downarrow})^{-1}. \]
Therefore,
\begin{align*}
\widehat{\mathcal{T} }  & = (\rho^*_{\downarrow})^{-1} \circ \phi^*\vert_{\Fix(N)} \circ \tau\vert_{\Fix(N)} \circ \rho^*_{\downarrow} \\
& =  (\hat{\phi})^* \circ (\rho^*_{\downarrow})^{-1} \circ \tau\vert_{\Fix(N)} \circ \rho^*_{\downarrow}.
\end{align*}
By \cite[Prop. 1.6.1]{CSC10}, $ (\rho^*_{\downarrow})^{-1} \circ \tau\vert_{\Fix(N)} \circ \rho^*_{\downarrow}$ is an $\id$-cellular automaton, so $\widehat{\mathcal{T} }$ is indeed a $\hat{\phi}$-cellular automaton. 

To show the uniqueness, assume that $\mathcal{S} : A^{G/N} \to A^{G/N}$ makes the diagram commute. Then, 
\[ \rho^*\circ \mathcal{S} = \mathcal{T} \circ \rho^* = \rho^*\circ \widehat{\mathcal{T} }.  \]
As $\rho^*$ is injective, it is left-cancellative, so $\mathcal{S} = \widehat{\mathcal{T} }$.
\end{proof}

\begin{corollary}\label{quotient}
Let $N$ be a fully invariant subgroup of $G$. The map $\mathcal{T} \mapsto \widehat{\mathcal{T} }$ from $\GCA(A^G)$ to $\GCA(A^{G/N})$ is a monoid homomorphism. 
\end{corollary}
\begin{proof}
The map $\mathcal{T} \mapsto \widehat{\mathcal{T} }$ is a monoid homomorphism because it is induced by conjugation by $\rho^*_{\downarrow} : A^{G/N} \to \Fix(N)$.
\end{proof}

\begin{remark}
Unlike Proposition 1.6.2 in \cite{CSC10}, where a monoid epimorphism from $\CA(A^G)$ to $\CA(A^{G/N})$ is defined, the homomorphism from $\GCA(A^G)$ to $\GCA(A^{G/N})$ of Corollary \ref{quotient} may not be surjective since the homomorphism $\End(G) \to \End(G/N)$ given by $\phi \mapsto \hat{\phi}$ may not be surjective. 
\end{remark}

%%%%%%%%%%%%%%%%%%%%%%%%%%%%%%%%%%%%%%

\section{Induction and restriction}

This section is inspired by the induction and restriction of classical cellular automata (see \cite{CSC09} and \cite[Sec. 1.7]{CSC10}). 

For a subgroup $K \leq H$ and $\phi\in \Hom(H,G)$, let $\mathcal{T}: A^G\to A^{H}$ be a $\phi$-cellular automaton with memory set $T\subseteq \phi(K)$ and local function $\mu:A^{T}\to A$. Let $\phi \vert_K^{\phi(K)} : K \to \phi(K)$ be the domain and codomain restriction of $\phi$. Define the \emph{restriction} of $\mathcal{T}$ with respect to $K$ as the $\phi \vert_K^{\phi(K)}$-cellular automaton $\mathcal{T}_{K}: A^{\phi(K)}\to A^{K}$  defined by
\[ \mathcal{T}_{K}(x)(k):=\mu((\phi(k^{-1}) \cdot x)|_{T }), \]
for all $x \in A^{\phi(K)}$ and $k\in K$. 

\begin{lemma}\label{le-restriction}
With the above notation, $\mathcal{T}_K$ is the unique $\phi \vert_K^{\phi(K)}$-cellular automaton such that the following diagram commutes
\[ \begin{tikzcd}
A^{G} \arrow{r}{\mathcal{T}} \arrow[swap]{d}{\Res_{\phi(K)}} & A^{H} \arrow{d}{\Res_K} \\%
A^{\phi(K)} \arrow{r}{\mathcal{T}_K}& A^K
\end{tikzcd}\]
where $\Res_{\phi(K)}$ and $\Res_K$ are the natural restriction maps. 
\end{lemma}
\begin{proof}
Let $\mu : A^T \to A$ be the local function of $\mathcal{T}$. For all $x \in A^G$ and $k \in K$ we have,
\[ ( \mathcal{T}(x) \vert_K)(k) = \mu( (\phi(k^{-1}) \cdot x) \vert_T) = \mathcal{T}_K( x \vert_{\phi(K)})(k), \]
so the diagram commutes. To show uniqueness, suppose that $\mathcal{T}^\prime : A^{\phi(K)} \to A^K$ makes the diagram commute. For any $y \in A^{\phi(K)}$, consider an extension $\hat{y} \in A^G$. Then, for all $k \in K$,
\[\mathcal{T}^\prime(y)(k) = \mathcal{T}^\prime( \hat{y} \vert_{\phi(K)})(k) = (\mathcal{T}(\hat{y}) \vert_K )(k) = \mathcal{T}_K(\hat{y} \vert_{\phi(K)})(k) = \mathcal{T}_K(y)(k). \]
Therefore, $\mathcal{T}^\prime = \mathcal{T}_K$. 
\end{proof}

Now, for $\phi \in \Hom(H,G)$, consider a $\phi \vert_K^{\phi(K)}$-cellular automaton $\mathcal{S} : A^{\phi(K)}\to A^{K}$ with memory set $S\subseteq \phi(K)$ and local function $\nu : A^{S}\to A$. Define the \emph{induction} of $\sigma$ with respect to $H$ as the $\phi$-cellular automaton $\mathcal{S} ^{H}: A^{G}  \to  A^{H}$ defined by
\[ \mathcal{S} ^{H}(x)(h):=\nu ((\phi(h^{-1}) \cdot x)|_{S}), \]
for all $x\in A^{G}$ and $g\in G$. 

\begin{lemma}
With the above notation, $\mathcal{S}^H$ is the unique $\phi$-cellular automaton such that the following diagram commutes
\[ \begin{tikzcd}
A^{G} \arrow{r}{\mathcal{S}^H} \arrow[swap]{d}{\Res_{\phi(K)}} & A^{H} \arrow{d}{\Res_K} \\%
A^{\phi(K)} \arrow{r}{\mathcal{S}}& A^K
\end{tikzcd}\]
\end{lemma}
\begin{proof}
This is analogous to the proof of Lemma \ref{le-restriction}.
\end{proof}

When $G=H$ and $\phi = \id$, the restriction and induction of $\id$-cellular automata coincide with the restriction and induction defined in \cite[Sec. 1.7.]{CSC10}.

The proof of the next lemma follows immediately from the definition of restriction and induction.

\begin{lemma}\label{lemma:(tau|_K)^(G,H)=tau}
Let $K \leq H$ be a subgroup and $\phi\in \Hom(H,G)$.
\begin{enumerate}
\item For every $\phi$-cellular automaton $\mathcal{T}: A^G\to A^{H}$ with memory set contained in $\phi(K)$,
\[ (\mathcal{T}_{K})^{H}=\mathcal{T}. \]

\item For every $\phi \vert_K^{\phi(K)}$-cellular automaton $\mathcal{S}: A^{\phi(K)} \to A^{K}$, 
\[ (\mathcal{S}^{H})_{K}= \mathcal{S}. \]
\end{enumerate}
\end{lemma}

\begin{proposition}\label{indres-comp}
Let $G$, $H$, and $R$ be groups, and let $\phi \in \Hom(H,R)$ and $\psi \in \Hom(R,G)$. Let $K \leq H$ be a subgroup. Let $\mathcal{T}: A^{R} \to A^{H}$ be a $\phi$-cellular automaton with memory set $T \subseteq \phi(K)$ and let $\mathcal{S} : A^{G} \to A^{R}$ be a $\psi$-cellular automaton with memory set contained in $S \subseteq \psi(\phi(K))$. Then,
\[ (\mathcal{T} \circ\mathcal{S})_{K}=\mathcal{T}_{K}\circ \mathcal{S}_{\phi(K)}. \]
\end{proposition}

\begin{proof}
By Theorem \ref{th-old} (2.), $\mathcal{T} \circ\mathcal{S} : A^G \to A^H$ is a $(\psi \circ \phi)$-cellular automaton with memory set $\psi(T)S$. Since $T \subseteq \phi(K)$, then $\psi(T) \subseteq \psi(\phi(K))$, so $\psi(T)S \subseteq \psi (\phi(K))$, as $\psi(\phi(K))$ is a subgroup of $R$. Hence, we may consider the restriction $(\mathcal{T} \circ\mathcal{S})_{K} : A^{\psi (\phi(K))} \to A^K$. 

For each $x \in A^{\psi(\phi(K))}$, consider an extension $\hat{x} \in A^G$. Then, for all $k \in K$ and $x \in A^{\psi(\phi(K))}$, we have
\[ (\mathcal{T} \circ\mathcal{S})_{K}(x)(k)=(\mathcal{T} \circ\mathcal{S})(\hat{x})(k). \]
On the other hand, for all $x \in  A^{\psi(\phi(K))}$, note that $\mathcal{S}(\hat{x}) \in A^R$ is an extension of $\mathcal{S}_{\phi(K)}(x) \in A^{\phi(K)}$, and $\mathcal{T}(\mathcal{S}(\hat{x})) \in A^H$ is an extension of $\mathcal{T}_{K}(\mathcal{S}_{\phi(K)}(x)) \in A^K$. Hence, for all $k \in K$ and $x \in A^{\psi(\phi(K))}$,
\[ \mathcal{T}_{K}(\mathcal{S}_{\phi(K)}(x))(k)=\mathcal{T}(\mathcal{S}(\hat{x}))(k). \]
This shows that for all $k \in K$ and $x \in A^{\psi(\phi(K))}$,
\[  (\mathcal{T} \circ \mathcal{S})_{K}(x)(k) = \mathcal{T}_{K}(\mathcal{S}_{\phi(K)}(x))(k). \]
The result follows. 
\end{proof}

\begin{lemma}
Let $\phi \in \Hom(H,G)$ and $K\leq H$. Then,
\[ \phi^*_K = (\phi \vert_K^{\phi(K)})^*, \]
where $\phi \vert_K^{\phi(K)} : K \to \phi(K)$ is the domain and codomain restriction of $\phi$. 
\end{lemma}
\begin{proof}
For any $x \in A^{\phi(K)}$ and any $k \in K$, we have
\[ \phi^*_K(x) (k) = x( \phi(k)) = x( \phi \vert_K^{\phi(K)}(k)) = (\phi \vert_K^{\phi(K)})^*(x)(k).  \] 
\end{proof}

\begin{corollary}\label{cor-phi-res}
Let $\phi \in \Hom(H,G)$ and $K\leq H$. 
\begin{enumerate}
\item $\phi_K^* : A^{\phi(K)} \to A^K$ is injective.
\item If $\phi^*$ is surjective, then $\phi_K^*$ is surjective. 
\item If $\ker(\phi) \leq K$ and $\phi_K^*$ is surjective, then $\phi^*$ is surjective. 
\end{enumerate}
\end{corollary}
\begin{proof}
\begin{enumerate}
\item As $\phi \vert_K^{\phi(K)} : K \to \phi(K)$ is surjective, then $\phi^*_K = (\phi \vert_K^{\phi(K)})^*$ is injective by Lemma \ref{le-star}. 

\item If $\phi^*$ is surjective, then $\phi : H \to G$ is injective by Lemma \ref{le-star}. Then, $\phi \vert_K^{\phi(K)} : K \to \phi(K)$ is also injective, so $\phi^*_K = (\phi \vert_K^{\phi(K)})^*$ is surjective. 

\item If $\phi_K^*$ is surjective, then $\phi \vert_K^{\phi(K)} : K \to \phi(K)$ is injective. Since $\ker(\phi) \leq K$, we have
\[ \{ e\} = \ker \left( \phi \vert_K^{\phi(K)}  \right) = \ker(\phi) \cap K = \ker(\phi). \]
This shows that $\phi : H \to G$ is injective, so $\phi^*$ is surjective by Lemma \ref{le-star}.
\end{enumerate}
\end{proof}

\begin{theorem}
Let $K \leq H$ and $\phi \in \Hom(H,G)$. Let $\mathcal{T} : A^G \to A^H$ be a $\phi$-cellular automaton with memory set contained in $\phi(K)$. 
\begin{enumerate}
\item $\mathcal{T}$ is injective if and only if $\mathcal{T}_K$ is injective and $\phi$ is surjective.
\item $\mathcal{T}$ is bijective if and only if $\mathcal{T}_K$ is bijective and $\phi$ is bijective.  
\end{enumerate} 
\end{theorem}
\begin{proof}
By Lemma \ref{lemma:GCA-tau_phi}, there exists $\tau \in \CA(A^G)$ such that $\mathcal{T} = \phi^* \circ \tau$. By Proposition \ref{indres-comp}, 
\[ \mathcal{T}_K = \phi^{\ast}_K \circ \tau_{\phi(K)}. \]
\begin{enumerate}
\item Suppose that $\mathcal{T}$ is injective. By Lemma \ref{le-phi}, $\phi$ is surjective.  The injectivity $\mathcal{T}=\phi^* \circ \tau$ implies that $\tau$ is injective, and $\tau_{\phi(K)}$ is injective by \cite[Prop. 1.7.4]{CSC10}. As $\phi_K^*$ is always injective by Corollary \ref{cor-phi-res}, then $\mathcal{T}_K =  \phi_K^* \circ \tau_{\phi(K)}$ is also injective. 

Conversely, suppose that $\mathcal{T}_K =  \phi_K^* \circ \tau_{\phi(K)}$ is injective and $\phi$ is surjective. Then $\tau_{\phi(K)}$ is injective, and $\tau$ is injective by \cite[Prop. 1.7.4]{CSC10}. Moreover, $\phi^*$ is injective by Lemma \ref{le-star}, so $\mathcal{T} = \phi^* \circ \tau$ is injective. 

\item This proof is analogous to the proof of part (1.) and uses the fact that $\tau$ is bijective if and only if $\tau_{\phi(K)}$ is bijective \cite[Prop. 1.7.4]{CSC10}. 
\end{enumerate}
\end{proof}

\begin{question}
What can we say about the relationship of between the surjectivity of $\mathcal{T}$ and the surjectivity of $\mathcal{T}_K$? 
\end{question}

%%%%%%%%%%%%%%%%%%%%%%%%%%%%%%%%%

\section*{Acknowledgments}

We sincerely thank the anonymous reviewer of this paper for their careful reading and insightful comments; in particular, they greatly simplified the proof of Lemma \ref{le-power}. The second  author was supported by a CONAHCYT Postdoctoral Fellowship \emph{Estancias Posdoctorales por M\'exico}, No. I1200/320/2022.

%%%%%%%%%%%%%%%%%%%%%%%%%%%%%%%%%%%%%%%%%%%%%%%%%%%%%%%%%%%%%%%%%%%%%%%%%%%%%%%%%%%%%%%%%%%%%

\end{document}